\documentclass{amsart}

\usepackage[all]{xy}

\usepackage{cleveref}

\theoremstyle{plain}
\newtheorem{theorem}[subsection]{Theorem}
\newtheorem{proposition}[subsection]{Proposition}
\newtheorem{lemma}[subsection]{Lemma}
\newtheorem{corollary}[subsection]{Corollary}
\theoremstyle{definition}
\newtheorem{definition}[subsection]{Definition}
\newtheorem{example}[subsection]{Example}

\newcommand{\codim}{\textrm{codim}}
\newcommand{\F}{\mathcal{F}}
\newcommand{\Id}{\textrm{Id}}
\renewcommand{\Im}{\textrm{Im}}
\newcommand{\Ker}{\textrm{Ker}}
\newcommand{\lcm}{\textrm{lcm}}
\newcommand{\length}{\textrm{length}}
\newcommand{\N}{\mathbb{N}}
\renewcommand{\O}{\mathcal{O}}
\newcommand{\pr}{\textrm{pr}}
\newcommand{\Q}{\mathbb{Q}}
\newcommand{\red}{\textrm{red}}
\newcommand{\Spec}{\textrm{Spec}}
\newcommand{\Supp}{\textrm{Supp}}
\newcommand{\Z}{\mathbb{Z}}

\def\cl[#1]{\overline{\{#1\}}}
\newcommand{\cycl}{\textrm{cycl}}
\newcommand{\desc}{\textrm{desc}}
\newcommand{\eff}{\textrm{eff}}
\newcommand{\naive}{{*,\textrm{naive}}}
\newcommand{\res}{\textrm{res}}
\newcommand{\tr}{\textrm{tr}}
\newcommand{\Zcal}{\mathcal{Z}}

\begin{document}

\title{Descent of algebraic cycles}
\author{Johannes Ansch\"{u}tz}
\date{08.06.2015}

\begin{abstract}
We characterize universally generalizing morphisms which satisfy descent of algebraic cycles integrally as those universally generalizing morphisms which are surjective with generically reduced fibres.
In doing so, we introduce a naive pull-back of cycles for arbitrary morphisms between noetherian schemes, which generalizes the classical pull-back for flat morphisms, and then prove basic properties of this naive pull-back.
\end{abstract}

\maketitle

\section{Introduction}

\noindent
These notes discuss general descent properties of algebraic cycles.
The basic question can be described as follows.
Let $X,Y$ be noetherian schemes with groups of algebraic cycles $\Zcal^*(X),\Zcal^*(Y)$ and let $f:X\to Y$ be a surjective morphism.
We ask ourselves which conditions on $f$ guarantee that the \textit{descent sequence}
\begin{equation}
\label{equation-descent-sequence-in-the-introduction}
0\longrightarrow\mathcal{Z}^*(Y)\overset{f^*}{\longrightarrow}\mathcal{Z}^*(X)\overset{\pr_1^*-\pr_2^*}{\longrightarrow} \mathcal{Z}^*(X\times_Y X)
\end{equation}
is exact.
Inspired by known descent theory assuming $f$ faithfully flat seems natural, but turns out to be insufficient.
The obstruction is given by a (super)natural number $g_Y(f)$ which is defined as follows.
For $y\in Y$ let
$$
g_y(f):=\gcd\{\ \length(\O_{f^{-1}(\cl[y]),x})\ |\ x\in f^{-1}(\cl[y])\ \text{generic}\},
$$
then set
$$
g_Y(f):=\lcm\{\ g_y(f)\ |\ y\in Y\}.
$$
Paying the prize of introducing a naive pull-back of cycles for arbitrary morphisms (see \Cref{definition-naive-pull-back}), flatness can be replaced by the weaker notion of a universally generalizing morphism. A morphism $f:X\to Y$ of schemes is called generalizing if for every $x\in X$ the induced morphism
$$
f:\Spec(\O_{X,x})\to \Spec(\O_{Y,f(x)})
$$
is surjective (\cite[D\'{e}finition (1, 3.9.2)]{EGA_I_second}). We call $f$ universally generalizing if it stays generalizing after every base change. Typical examples of universally generalizing morphisms are flat morphisms.
We obtain the following answer to our question about descent of cycles:

\begin{theorem}
\label{theorem-main-theorem-introduction}
Assume $f:X\to Y$ is a surjective universally generalizing morphism of noetherian schemes such that $X\times_Y X$ is again noetherian.
Then the sequence \eqref{equation-descent-sequence-in-the-introduction} has torsion cohomology which vanishes if and
only if $\ g_Y(f)=1$.
\end{theorem}

\medskip\noindent
In particular, we obtain that descent of cycles holds rationally for arbitrary surjective universally generalizing morphisms $f:X\to Y$ between noetherian schemes such that $X\times_Y X$ is noetherian.
Examples of morphisms which satisfy descent integrally are surjective \textit{smooth} morphisms.

\medskip\noindent
We start these notes by recalling the construction of cycles associated to subschemes and define a naive pull-back of cycles.
The construction of the naive pull-back, as we present it here, reveils basic problems.
For example, the assignment $X\mapsto \mathcal{Z}^*(X)$ is not functorial for \textit{all} scheme morphisms, only for \textit{flat} ones (see \Cref{example-naive-pull-back-not-functorial}).
After discussing the push-forward of cycles along closed immersions, we will prove our main theorem \Cref{theorem-main-theorem-introduction} resp.\ \Cref{theorem-main-theorem}.
As we try to be very general we also discuss the rather trivial case of descent along a universally bijective morphism (see \Cref{proposition-descent-along-universally-bijective-morphisms}).

\section{Cycles of subschemes and naive pull-back of cycles}

\noindent
Let $X$ be a scheme.
We denote by
$$
\mathcal{Z}^*(X):=\bigoplus\limits_{x\in X} \Z x
$$
the free abelian group generated by the set underlying the topological space of the scheme $X$ and call elements in $\mathcal{Z}^*(X)$ cycles on $X$. If the local ring $\mathcal{O}_{X,x}$ has finite Krull dimension for every $x\in X$, e.g.\ if $X$ is locally noetherian, the group $\mathcal{Z}^*(X)$ is naturally graded by setting
$$
\mathcal{Z}^r(X):=\bigoplus\limits_{\substack{x\in X,\\ \textrm{ codim}(\overline{\{x\}},X)=r}}\mathbb{Z} x.$$
We want to associate cycles to closed subschemes.
Defining more generally cycles associated to coherent modules turns out to be more flexible. Before giving the definition we recall that the support $\Supp(\F)$ of a coherent sheaf $\F$ on a noetherian scheme $X$ is the closed subscheme of $X$ defined by the annihilator
$$
\textrm{Ann}_{\O_X}(\F):=\Ker(\O_X\to \underline{\textrm{Hom}}(\F,\F))
$$
of $\F$.

\begin{definition}
Let $X$ be a noetherian scheme and let $\F$ be a coherent $\O_X$-module on $X$.
We define
$$
\cycl(\F):=\sum\limits_{\substack{x\in \Supp(\F)\\ \text{generic point}}} \length_{\mathcal{O}_{X,x}}(\F_x)\cdot x \in \Zcal^*(X).
$$
If $Z\subseteq X$ is a closed subscheme, then we set
$$
\cycl(Z):=\cycl(\O_Z)=\sum\limits_{\substack{x\in Z\\ \text{generic point}}} \length(\O_{Z,x})\cdot x.
$$
\end{definition}

\noindent
The noetherianess assumption is needed at two places.
Firstly, that $Z$ has only finitely many generic points and secondly to assure that the lengths at these generic points are finite.
In general, these zero-dimensional local rings need not be artinian.
We see that the theory of cycles (in our definition above) is from the start restricted to \textit{noetherian} schemes.

\begin{lemma}
\label{lemma-cycles-of-sheaves-are-additive-in-short-exact-sequences}
Let $X$ be a noetherian scheme and let
$$
0\to \F\to \mathcal{G}\to \mathcal{H}\to 0
$$
be a short exact sequence of coherent $\O_X$-modules such that $\Supp(\F)=\Supp(\mathcal{H})$.
Then
$$
\cycl(\mathcal{G})=\cycl(\F)+\cycl(\mathcal{H}).
$$
\end{lemma}

\begin{proof}
As $\Supp(\mathcal{G})=\Supp(\F)\cup\Supp(\mathcal{H})$ we can conclude
$$
\Supp(\mathcal{G})=\Supp(\F)=\Supp(\mathcal{H}),
$$
and hence the points where one of the sheaves $\F, \mathcal{G},\mathcal{H}$ is of finite length are precisly the generic points of $\Supp(\mathcal{G})$.
As lengths are additive in short exact sequences, we can conclude.
\end{proof}

\begin{definition}
\label{definition-naive-pull-back}
Let $f:X\to Y$ be a morphism of noetherian schemes.
We define the naive pull-back
$$
f^{*,\textrm{naive}}:\mathcal{Z}^*(Y)\to\mathcal{Z}^*(X)
$$
by linear extension of the map
$$
y \in Y\mapsto \cycl(f^{-1}(\cl[y]))\in \mathcal{Z}^*(X)
$$
where $f^{-1}(\cl[y])$ denotes the scheme-theoretic pull-back of the closed reduced subscheme $\cl[y]\subseteq Y$.
\end{definition}

\medskip\noindent
As we will see this definition lacks some properties in the case $f$ is not flat.
Firstly, for $y\in Y$ the generic points of $f^{-1}(\cl[y])$ need not lie over $y$ as examples of closed immersions show.
This problem can be solved by requiring that $f$ is \textit{generalizing}.

\begin{definition}
A morphism $f:X\to Y$ of schemes is called generalizing if for all $x\in X$ the induced morphism
$$f:\Spec(\O_{X,x})\to\Spec(\O_{Y,f(x)})$$ is surjective.
Moreover, $f$ is called universally generalizing, if every base change of $f$ is generalizing.
\end{definition}

\medskip\noindent
If $f:X\to Y$ is generalizing, then for all $y\in Y$ the generic points of $f^{-1}(\cl[y])$ lie over $y$.
This property turns out to be of fundamental importance for our question. Examples of generalizing morphisms are flat morphisms or open morphisms.
Conversely, generalizing morphisms which are locally of finite presentation are open.
Another source of examples for universally open morphisms, usually not flat, can be given as follows.

\begin{example}
\label{example-normalisation-yields-universally-generalizing-morphisms} 
Assume that $f:X\to Y$ is a finite morphism, which is the normalisation of a geometrically unibranch integral noetherian scheme $Y$ inside a finite field extension at the generic point of $Y$.
It is known that $f$ satisfies the going-down theorem, i.e., $f$ is generalizing.
As $f$ is of finite type and $Y$ noetherian, $f$ is an open morphism.
Finally, \cite[Corollaire (14.4.3)]{EGA.IV.3} shows that $f$ is even universally open as $Y$ is geometrically unibranch.
\end{example}

\medskip\noindent
The definition of $f^{*,\textrm{naive}}$ involves the calculations of lengths, hence a purely topological condition like generalizing can not be sufficient in the following proposition.

\begin{proposition}
\label{proposition-flat-pull-back-commutes-with-cycles-of-subschemes}
 Let $f:X\to Y$ be a flat morphism of noetherian schemes and let $Z\subseteq Y$ be a closed subscheme. Then
 $$f^{*,\textrm{naive}}(\textrm{cycl}(Z))=\textrm{cycl}(f^{-1}(Z)).$$
\end{proposition}
\begin{proof}
This is proven in \cite[Lemma 1.7.1]{IntersectionTheory}.
\end{proof}

\medskip\noindent
The necessity of flatness in \Cref{proposition-flat-pull-back-commutes-with-cycles-of-subschemes} has consequences for the functoriality of $f\mapsto f^{*,\textrm{naive}}$ and is the main problem encountered with the naive pull-back of cycles in the absence of flatness.

\begin{proposition}
\label{proposition-functoriality-for-pull-back}
Let $f:X\to Y$ and $g:Y\to Z$ be morphisms of noetherian schemes. Assume that $f$ is flat. Then
$$(g\circ f)^{*,\textrm{naive}}=f^{*,\textrm{naive}}\circ g^{*,\textrm{naive}}.$$
\end{proposition}

\begin{proof}
Let $z\in Z$ be a point. Then
$$(g\circ f)^{*,\textrm{naive}}(z)=\textrm{cycl}(f^{-1}g^{-1}(\overline{\{z\}}))$$
while
$$f^{*,\textrm{naive}}\circ g^{*,\textrm{naive}}(z)=f^{*,\textrm{naive}}(\textrm{cycl}(g^{-1}(\overline{\{z\}})).$$
Both cycles agree by \Cref{proposition-flat-pull-back-commutes-with-cycles-of-subschemes} as $f$ is flat.
\end{proof}

\medskip\noindent
As an example that flatness of $f$ is really needed in \Cref{proposition-functoriality-for-pull-back} one can take $g:Y\to Z=\Spec( k)$ a (noetherian) scheme over a field and $f:X=Y_\red\to Y$ the natural closed immersion.
Then in general
$$
\begin{matrix}
(g\circ f)^\naive(\cycl(Z))=\cycl(Y_\red)\\[1em]
\neq f^\naive(\cycl(Y))=f^\naive(g^\naive(\cycl(Z))).
\end{matrix}
$$
An example with smooth schemes is the following.

\begin{example}
\label{example-naive-pull-back-not-functorial}
We take $Y=\Spec(k[t,x]/(x^n-t))$ with $n\geq 2,$ $X=\Spec(k[t]/(t))=\mathrm{pt}$, $Z=\Spec(k[t])$ and as morphisms
$$
\begin{matrix}
f:X\to Y: t\mapsto (t,0)\\
g:Y\to Z: (t,x)\mapsto t.
\end{matrix}
$$
Let $\Zcal:=\cycl(k[t]/(t))\in \Zcal^*(Z)$ which is a cycle on $Z$.
Then
$$
(g\circ f)^\naive(\Zcal)=\cycl(k[t]/(t))\in \Zcal^*(X)
$$
while
$$
f^\naive(g^\naive(\Zcal))=n\cdot\cycl(k[t]/(t)).
$$
\end{example}

\Cref{proposition-functoriality-for-pull-back} is even wrong if $g$ is flat and $f$ a universal homeomorphism, therefore in particular generalizing.
We give an example.

\begin{example}
Let $X=\Spec(k[t])$ be the affine line over a field $k$ and $Y=\Spec(k[t^2,t^3])$, which is a curve with a cusp at the ideal $(t^2,t^3)$ and normalisation $X$.
We take $Z=\Spec(k[t^2])$ and $f:X\to Y$, $g:Y\to Z$ the morphisms given by the inclusions
$$
k[t^2]\subseteq k[t^2,t^3]\subseteq k[t].
$$
Then $g$ is flat and $f$ a universal homeomorphism.
Let
$$
\Zcal:=\cycl(k[t^2]/{(t^2)})\in \Zcal^*(Z).
$$
Then
$$
g^\naive(\Zcal)=\cycl(k[t^2,t^3]/(t^2))=2\cdot\cycl(k[t^2,t^3]/(t^2,t^3))
$$
hence
$$
f^\naive(g^\naive(\Zcal))=4\cdot\cycl(k[t]/(t)).
$$
On the other hand,
$$
(g\circ f)^\naive(\Zcal)=\cycl(k[t]/(t^2))=2\cdot\cycl(k[t]/(t)).
$$
This example can also be modified to show that the naive pull-back does not preserve rational equivalence, even for universal homeomorphisms.
In fact, the morphism $f:X\to Y$ extends uniquely to the canonical compactifications $X'$, $Y'$ of $X,Y$ giving a morphism $f:X'\to Y'$.
If we denote by $\infty$ the boundary of $X\subseteq X'$, resp.\ $Y\subseteq Y'$, then $f(\infty)=\infty$.
The zero divisor of $t^2\in k[t^2,t^3]$ is
$$
\Zcal:=2\cdot\cycl(k[t^2,t^3]/(t^2,t^3))-2\cdot\infty
$$
which has a pull-back
$$
f^\naive(\Zcal)=4\cdot\cycl(k[t]/(t))-2\cdot\infty
$$
not rationally equivalent to zero on $X'\cong \mathbb{P}^1_k$.
\end{example}

\medskip\noindent
\Cref{proposition-functoriality-for-pull-back} shows that the assignment $f\mapsto f^{*,\textrm{naive}}$ is functorial for \textit{flat} maps between noetherian schemes.
For a flat map $f$ we therefore abbreviate $f^{*,\textrm{naive}}$ by $f^*$ as this is the pull back of cycles usually encountered (for example in \cite{IntersectionTheory}).
However, in \Cref{proposition-functoriality-for-generalizing-and-generically-reduced-morphisms} and \Cref{lemma-naive-pull-back-for-cartesian-diagram} we investigate situations guaranteing functoriality for the naive pull-back.

\begin{proposition}
\label{proposition-f-surjective-implies-f*-injective}
Let $f:X\to Y$ be a surjective morphism of noetherian schemes.
Then
$$
f^{*,\textrm{naive}}:\mathcal{Z}^*(Y)\to\mathcal{Z}^*(X)
$$
is injective.
If in addition $f$ is generalizing, then conversely injectivity of $f^\naive$ implies surjectivity of $f$.
\end{proposition}

\begin{proof}
We note that under the assumption that $f$ is generalizing the morphism $f^\naive$ is the direct sum over $y\in Y$ of the maps
$$
f^\naive_{|\Z y}:\Z y\to \bigoplus\limits_{x\in f^{-1}(y)}\Z x
$$
because for $y\in Y$ the generic points of the preimage $f^{-1}(\overline{\{y\}})$ all lie over $y$.
In particular, $f^\naive$ is injective if and only if each $f^\naive_{|\Z y}$ is injective.
But $f^\naive_{\Z y}$ is injective if and only if $f^{-1}(y)\neq \emptyset$.
Hence, all statements are in fact obvious if $f$ is generalizing. This said, we now turn to the general case by assuming that $f$ is only surjective, but not necessarily generalizing.
Let
$$
\Zcal:=\sum\limits_{i=1}^n m_iy_i\in \Zcal^*(Y)
$$
be a cycle with $f^\naive(\Zcal)=0$.
We may assume that the $y_i$ are pairwise distinct and $m_i\neq0$ for every $i$.
Moreover, we can arrange
$$
\codim(\cl[y_1],Y)\leq\ldots\leq\codim(\cl[y_{n}],Y).
$$
Let $x_1\in f^{-1}(y_1)$ be a generic point of the fibre $f^{-1}(y_1)$ (which exists because $f$ is surjective).
Because
$$
-m_1f^*(y_1)=\sum\limits_{i=2}^n m_if^*(y_i)
$$
there exists some $y_j,j\geq 2,$ such that $x_1$ is a generic point of $f^{-1}(\overline{\{y_j\}})$.
In particular, $y_1=f(x_1)$ is a specialisation of $y_j$ and as $\textrm{codim}(\overline{\{y_1\}},Y)\leq \textrm{codim}(\overline{\{y_j\}},Y)$ we get $y_j=y_1$, a contradiction.
\end{proof}

\medskip\noindent
We give an example, that the injectivity of $f^\naive$ does not necessarily imply that $f$ is surjective.

\begin{example}
\label{example-f*-injective-does-not-imply-f-surjective}
Let $Y:=\Spec(R)$ be the spectrum of a discrete valuation ring $R$ with residue field $k:=R/{\mathfrak{m}}$ and denote by $\eta, s=\Spec(k)$ its generic resp.\ special point.
Consider the schemes $X_1:=\Spec(R/{\mathfrak{m}^n})$ and $X_2:=s$ and let $x_i\in X_i$ be their unique points.
We define $f:X=X_1\coprod X_2\to Y$ as the natural morphism.
Then
$$
\begin{matrix}
f^\naive(\eta)=nx_1+x_2\\
f^\naive(s)=x_1+x_2.
\end{matrix}
$$
In particular, for $n\geq 2$ both cycles are linearly independant.
Hence,
$$
f^\naive:\Zcal^*(Y)\to \Zcal^*(X)$$
is injective although $f$ is not surjective.
\end{example}

\medskip\noindent
In general, the naive pull-back does not preserve the natural grading of $\Zcal^*(X)$ given by codimension.
With the notation of example \ref{example-f*-injective-does-not-imply-f-surjective}, one can take the natural morphism $s\coprod Y\to Y$.
More serious examples include blow-ups, for example that of points on surfaces.
However, in the case of a generalizing morphism this problem disappears as will be proven in \Cref{proposition-generalizing-morphisms-preserve-grading-by-codimension}.

\begin{proposition}
\label{proposition-generalizing-morphisms-preserve-grading-by-codimension}
Let $f:X\to Y$ be a generalizing morphism of noetherian schemes and $y\in Y$.
Then
$$
\codim(\cl[y],Y)=\codim(\cl[x],X)
$$
for every generic point $x\in f^{-1}(y)$.
In particular, the homomorphism
$$
f^{\naive}:\mathcal{Z}^*(Y)\to \mathcal{Z}^*(X)
$$
respects the grading by codimension.
\end{proposition}

\begin{proof}
Let $A:=\O_{Y,y}$ resp.\ $B:=\O_{X,x}$ be the local rings at $y\in Y$ resp.\ $x\in X$ and let $\mathfrak{m}:=\mathfrak{m}_{Y,y}$ be the maximal ideals in $A$.
As $x\in f^{-1}(y)$ is a generic point, the ring
$$
B/{\mathfrak{m}B}
$$
is artinian.
More generally, the ring $B/{IB}$ is artinian for every $\mathfrak{m}$-primary ideal $I\subseteq A$. In other words, for every ideal of definition $I\subseteq A$, that is every $\mathfrak{m}$-primary ideal, the ideal $IB$ is an ideal of definition for $B$.
As the Krull dimension of any local noetherian ring can be computed as the minimal number of generators for ideals of definition, we can conclude
$$
\dim(B)\leq \dim(A).
$$
The morphism $f:\Spec(B)\to \Spec(A)$ is generalizing, hence every chain
$$
y_n\rightsquigarrow y_{n-1} \rightsquigarrow \ldots \rightsquigarrow y_0
$$
of specialisations in $\Spec(A)$ can be lifted to a chain
$$
x_n\rightsquigarrow x_{n-1} \rightsquigarrow \ldots \rightsquigarrow x_0
$$
of specialisations in $\Spec(B)$ with $f(x_i)=y_i$.
In particular, $\dim(B)\geq \dim(A)$ and therefore $\dim(B)=\dim(A)$.
By definition, $\dim(B)=\codim(\cl[x],X)$ resp.\ $\dim(A)=\codim(\cl[y],Y)$ and the proposition is proven.
\end{proof}

\medskip\noindent
In some easy cases functoriality can be checked directly.
\begin{lemma}
\label{lemma-descent-sequences-compatible-with-zariski-localization}
Let $Y$ be a noetherian scheme and let $f:X\to Z$ be a morphism of noetherian schemes over $Y$.
Let $g: Y'\to Y$ be an open immersion or the Zariski localization $Y'=\Spec(\O_{Y,y})$ at a point $y\in Y$.
Then the diagram
$$
\xymatrix{
\Zcal^*(Z) \ar[r]^{f^\naive}\ar[d]^{p^*}& \Zcal^*(X)\ar[d]^{q^*} \\
\Zcal^*(Z')\ar[r]^{{f'}^\naive} & \Zcal^*(X')
}
$$
commutes, where
$$
f':X':=X\times_Y Y'\to Z':=Z\times_Y Y'
$$
denote the base change of $f$ to $Y'$ and $p:X'\to X$ resp.\ $q:Z'\to Z$ the natural projections.
\end{lemma}

\begin{proof}
Let $z\in Z$ be a point.
If $z\notin Z'$, then the claim is immediate as $Z'$ is stable under generalizations.
We assume $z\in Z'$ and denote by $\cl[z]_Z$ the closure of $z$ in $Z$.
Then $\cl[z]_Z\cap Z'=\cl[z]_{Z'}$ is the (reduced) closure of $z$ in $Z'$ and therefore the diagram
$$
\xymatrix{
\cl[z]_{Z'} \ar[r]\ar[d] & \cl[z]_Z\ar[d] \\
Z'\ar[r]^p & Z
}
$$
is cartesian. Base change along $f:X\to Z$ yields the cartesian diagram
$$
\xymatrix{
{f'}^{-1}(\cl[z]_{Z'}) \ar[r]\ar[d] & f^{-1}(\cl[z]_Z\ar[d]) \\
X'\ar[r]^q & X
}
$$
and as $q:X'\to X$ is flat, we can conclude by \Cref{proposition-flat-pull-back-commutes-with-cycles-of-subschemes}
\begin{multline*}
q^*(f^\naive(z))=q^*(\cycl(f^{-1}(\cl[z]_Z)))=\cycl(q^{-1}(f^{-1}(\cl[z]_Z))) \\[0.5em]
= \cycl({f'}^{-1}(p^{-1}(\cl[z]_Z)))= \cycl({f'}^{-1}(\cl[z]_{Z'})))={f'}^\naive(z)={f'}^\naive(p^*(z)).
\end{multline*}
We did some abuse of notation to denote by $z$ the cycles $z\in \Zcal^*(Z)$ and $p^*(z)=z\in \Zcal^*(Z')$.
\end{proof}

\noindent
As a special case of \Cref{lemma-descent-sequences-compatible-with-zariski-localization} we note that if $f:X\to Y$ is a morphism of noetherian schemes such that $X\times_Y X$ is again noetherian and $g:Y'\to Y$ as in \Cref{lemma-descent-sequences-compatible-with-zariski-localization}, then the diagram
\begin{equation}
\label{equation-descent-sequence-compatible-with-zariski-localization}
\xymatrixcolsep{7pc}
\xymatrix{
\Zcal^*(Y) \ar[d]^{g^*}\ar[r]^{f^\naive} & \Zcal^*(X) \ar[d]^{{g'}^*}\ar[r]^{\pr_1^\naive-\pr_2^\naive} & \Zcal^*(X\times_Y X)\ar[d]^{(g'\times g')^*} \\
\Zcal^*(Y') \ar[r]^{{f'}^\naive} & \Zcal^*(X') \ar[r]^{{\pr'}_1^\naive-{\pr'}_2^\naive} & \Zcal^*(X'\times_{Y'} X')
}
\end{equation}
with obvious notations commutes.

\noindent\medskip
In \Cref{proposition-functoriality-for-generalizing-and-generically-reduced-morphisms} we will partly generalize \Cref{lemma-descent-sequences-compatible-with-zariski-localization}.

\section{Push forward of cycles for closed immersions}

\noindent
We will need some limited covariant functoriality of cycle groups.

\begin{definition}
Let $Y$ be a noetherian scheme and let $f:X\to Y$ be a closed immersion.
We define (following for example \cite{IntersectionTheory}) the push-forward along $f$ by
$$
f_*:\Zcal^*(X)\to \Zcal^*(Y):\sum m_xx\mapsto\sum m_xf(x).
$$
\end{definition}

\begin{proposition}
\label{proposition-push-forward-commutes-with-cycles-of-subschemes}
Let $Y$ be a noetherian scheme and let $f:X\to Y$ be a closed immersion.
Then for all closed subschemes $Z\subseteq X$ we have the equality
$$
f_*(\cycl(Z))=\cycl(f(Z)).
$$
\end{proposition}

\begin{proof}
The morphism $f:Z\to f(Z)$ is an isomorphism, hence generic points are mapped to each other.
As for such a generic point $z\in Z$ the lengths
$$
\length(\O_{Z,z})=\length(\O_{f(Z),f(z)})
$$
agree, the claim follows.
\end{proof}

\begin{proposition}
\label{proposition-push-forward-commutes-with-naive-pull-back}
Let $f:X\to Y$ be a morphism of noetherian schemes.
Then for all closed immersions $i:Z\to Y$ the diagram
$$
\xymatrix{
\Zcal^*(Z) \ar[d]^{{f'}^\naive}\ar[r]^{i_*} & \Zcal^*(Y) \ar[d]^{f^\naive} \\
\Zcal^*(f^{-1}(Z))\ar[r]^{i'_*} & \Zcal^*(X)
}
$$
is commutative, where $i':f^{-1}(Z)\to X$ resp.\ $f':f^{-1}(Z)\to Z$ denote the base change of $i$ resp.\ $f$.
\end{proposition}

\begin{proof}
Let $z\in Z$ be a point.
Then
$$
f^\naive(i_*(z))=\cycl(f^{-1}(\overline{\{i(z)\}}))=\cycl(i'({f'}^{-1}(\overline{\{z\}}))).
$$
By \Cref{proposition-push-forward-commutes-with-cycles-of-subschemes}, this equals
$$
i'_*(\cycl({f'}^{-1}(\cl[z])))=i'_*({f'}^\naive(z))
$$
and the proof is finished.
\end{proof}

\noindent
Recall that a morphism $f:X\to Y$ is called weakly immersive if $f$ is a homeomorphism onto its image and for each $x\in X$ the induced morphism $f:\O_{Y,f(x)}\to \O_{X,x}$ is surjective.

\begin{lemma}
\label{lemma-generalizing-stable-under-base-change-along-immersions}
Let $f:X\to Y$ be a generalizing morphism of schemes and assume that $g:Y'\to Y$ is weakly immersive. Then the base change 
$$
f':X':=X\times_YY'\to Y'
$$ 
is again generalizing.
\end{lemma}
\begin{proof}
Let $x'\in X'$ be a point over $y'=f'(x')$ and $x=g'(x')$, $g':X'\to X$. Let $y:=f(x)=g(y')$. If $I:=\Ker(\O_{Y,y}\to \O_{Y',y'})$ denotes the kernel, then 
$$
\O_{X',x'}\cong \O_{X,x}/{I\O_{X,x}}
$$
and hence the diagram
$$
\xymatrix{
\Spec(\O_{X',x'}) \ar[r]\ar[d] & \Spec(\O_{Y',y'})\ar[d] \\
\Spec(\O_{X,x}) \ar[r] & \Spec(\O_{Y,y})
}
$$
cartesian. In particular, the upper arrow is surjective as $f$ is generalizing, showing that $f'$ is generalizing, too.
\end{proof}

\noindent
The next result will not be used in the sequel but seems interesting in its own right.
Namely, we will partly generalize \Cref{proposition-functoriality-for-pull-back} to generalizing morphisms, which are generically reduced.

\begin{proposition}
\label{proposition-functoriality-for-generalizing-and-generically-reduced-morphisms}
Let $f:X\to Y$, $g:Y\to Z$ be morphisms of noetherian schemes. Assume that $f$ is generalizing and that $g$ is generalizing with generically reduced fibres. Then
$$
f^\naive\circ g^\naive= (g\circ f)^\naive.
$$
\end{proposition}
\begin{proof}
We have to show $f^\naive\circ g^\naive(\Zcal)=(g\circ f)^\naive(\Zcal)$ for every cycle $\Zcal\in \Zcal^*(Z)$. By \ref{lemma-generalizing-stable-under-base-change-along-immersions} we may assume, using \Cref{proposition-push-forward-commutes-with-naive-pull-back}, that $Z$ is integral and $\Zcal=\eta$ for the generic point $\eta\in Z$. By \Cref{lemma-descent-sequences-compatible-with-zariski-localization} (and \ref{lemma-generalizing-stable-under-base-change-along-immersions}) we may further assume $Z=\Spec(k(\eta))$ as all generic points in $Y$ and $X$ lie over $\eta$ by assumption.
We have to check an equality of cycles whose components all lie over generic points of $Y$. Hence, we may replace $Y$ by the disjoint union of the spectra of the local rings at  the generic points of $Y$, arriving at a situation, where $Y$ is the spectrum of a finite product of fields as $Y$ is generically reduced. We can conclude 
$$
g^\ast(\eta)=\cycl(Y)=\sum\limits_{y\in Y}y
$$ 
and hence
$$
f^\naive(g^\naive(\eta))=f^\naive(\cycl(Y))=\cycl(X)=(g\circ f)^\naive(\eta)
$$
as $X=\coprod\limits_{y\in Y}f^{-1}(y)$ with each $f^{-1}(y)$ open. Therefore, we are finished with the proof.
\end{proof}

\noindent
We can record another situation, where the naive pull-back of cycles does not encounter problems.

\begin{lemma}
\label{lemma-naive-pull-back-for-cartesian-diagram}
Consider a cartesian diagram
$$
\xymatrix{
X'\ar[r]^p\ar[d]^q & Y'\ar[d]^g \\
X\ar[r]^f & Y
}
$$
of morphisms of noetherian schemes with all morphisms generalizing. Then
$$
p^\naive\circ g^\naive = (g\circ p)^\naive=(f\circ q)^\naive =q^\naive\circ f^\naive.
$$
\end{lemma}
\begin{proof}
Take $y\in Y$. As $f,g,p,q$ are generalizing, we can, by \Cref{lemma-generalizing-stable-under-base-change-along-immersions}, \Cref{proposition-push-forward-commutes-with-naive-pull-back} and \Cref{lemma-descent-sequences-compatible-with-zariski-localization} assume that $Y=\Spec(k(y))$ is a point. Then $f$,$g$ and hence $p$,$q$ are flat and the claim follows from \Cref{proposition-functoriality-for-pull-back}.  
\end{proof}

\section{Descent of algebraic cycles}

\noindent
In this section let $f:X\to Y$ be a surjective morphism of noetherian schemes.
We assume that $X\times_Y X$ is again noetherian\footnote{which excludes for example pro\'etale covers like $\Spec(k^{\textrm{sep}})\to \Spec(k)$ for (most) fields $k$} and denote by
$$
\pr_i:X\times_Y X\to X,\ i=1,2,
$$
the two projections.

\medskip\noindent
In the absence of flatness the following lemma is not obvious.

\begin{lemma}
\label{lemma-descent-sequence-is-a-complex}
If $f$ is universally generalizing, the sequence of homomorphisms
\begin{equation}
\label{equation-descent-sequence-for-cycles}
\xymatrixcolsep{4pc}
\xymatrix{
0\ar[r] & \Zcal^*(Y)\ar[r]^{f^\naive} & \Zcal^*(X)\ar[r]^{\pr_1^\naive-\pr_2^\naive} &  \Zcal^*(X\times_Y X)
}
\end{equation}
is a complex.
\end{lemma}

\begin{proof}
This is a consequence of \Cref{lemma-naive-pull-back-for-cartesian-diagram}.
\end{proof}

\medskip\noindent
We give an example that \eqref{equation-descent-sequence-for-cycles} need not be a complex in general.

\begin{example}
Let $f:X\to Y$ be as in \Cref{example-f*-injective-does-not-imply-f-surjective}.
We compute
$$
\pr^\naive_i(f^\naive(\eta))=\pr^\naive_i(nx_1+x_2).
$$
for $i=1,2$.
By definition,
$$
\begin{matrix}
\pr^\naive_1(x_1)=(x_1,x_1)+(x_1,x_2)\\
\pr^\naive_1(x_2)=(x_2,x_1)+(x_2,x_2)
\end{matrix}
$$
while
$$
\begin{matrix}
\pr^\naive_2(x_1)=(x_1,x_1)+(x_2,x_1)\\
\pr^\naive_2(x_2)=(x_1,x_2)+(x_2,x_2).
\end{matrix}
$$

We get
$$
\pr^\naive_1(nx_1+x_2)=n(x_1,x_1)+n(x_1,x_2)+(x_2,x_1)+(x_2,x_2)
$$
and
$$
\pr^\naive_2(nx_1+x_2)=n(x_1,x_1)+n(x_2,x_1)+(x_1,x_2)+(x_2,x_2),
$$
which are different cycles if $n\geq 2$.
\end{example}

\medskip\noindent
This example looks puzzling if compared to the situation for \textit{subschemes}.
The transitivity of fibre products implies that for $Z\subseteq Y$ a closed subscheme the pull-backs $\pr_1^{-1}(f^{-1}(Z))$ and $\pr_2^{-1}(f^{-1}(Z))$ are always equal.
The reason that this is wrong in general for \textit{cycles} is that in the absence of flatness the schematic pull-back $\pr_i^{-1}$ ``kills lengths'' in $f^{-1}(Z)$ while the cycle pull-back $\pr^\naive_i$ does not.

\medskip\noindent
We give a rather trivial condition on $f$ which guarantees that \eqref{equation-descent-sequence-for-cycles} is a complex.

\begin{lemma}
\label{lemma-descent-sequence-is-a-complex-for-universally-bijective-morphisms}
Assume that $f$ is universally bijective.
Then
$$
0\to\Zcal^*(Y)\overset{f^\naive}{\to}\Zcal^*(X)\overset{\pr_1^\naive-\pr_2^\naive}{\to} \Zcal^*(X\times_Y X)
$$
is a complex.
In fact, $\pr_1^*=\pr_2^*$.
\end{lemma}
\begin{proof}
We proof directly that $\pr_1^*=\pr_2^*$.
First, we remark, that $\pr_1$ and $\pr_2$ are in fact \textit{homeomorphisms} as the diagonal $X\to X\times_Y X$ is a continuous inverse for both.
Let $x\in X$.
We calculate
$$
\begin{matrix}
\pr_1^*(x)=\length(k(x)\otimes_{\O_{Y,f(x)}} \O_{X,x})(x,x)\\[1em]
=\length(k(x)\otimes_{k(f(x))} k(x))(x,x)=\pr_2^*(x).
\end{matrix}
$$
\end{proof}

\medskip\noindent
From now on we assume that \eqref{equation-descent-sequence-for-cycles} is a complex and call elements in $$
\Zcal^*_\desc(f):=\Ker(\pr_1^\naive-\pr_2^\naive)
$$ ``cycles with descent datum'' and elements in
$$
\Zcal^*_{\eff.\desc}(f):=\Im(f^\naive)
$$ ``cycles with effective descent datum''\footnote{ This should not be confused with ``effective'' cycles with descent datum, i.e., those cycles with descent datum whose coefficients are non-negative.}.
We denote by
$$
H_f:=\Zcal^*_\desc(f)/{\Zcal^*_{\eff.\desc}(f)}
$$
the cohomology at $\Zcal^*(X)$ of the complex \eqref{equation-descent-sequence-for-cycles}, in other words the group of cycles with descent datum modulo the cycles with effective ones.

\medskip\noindent
We come to an (obvious) obstruction for the vanishing of $H_f$.
For $y\in Y$ let
$$
g_y(f):=\gcd\{\ \length(\O_{f^{-1}(\cl[y]),x})\ |\ x\ \textrm{generic point in }\ f^{-1}(\cl[y])\}
$$
and
$$
g^\res_y(f):=\gcd\{\ \length(\O_{f^{-1}(y),x})\ |\ x\ \text{generic point in}\ f^{-1}(y)\}.
$$

\noindent
We remark that $g_y(f)$ divides $g_y^\res(f)$ but both numbers can be different in general. However, they agree if all generic points of $f^{-1}(\cl[y])$ lie over $y$.
We let
$$
g_Y(f):=\lcm\{\ g_y(f)\ |\ y\in Y\}
$$
be the least common multiple of the $g_y(f)$, which we understand as a supernatural number, i.e., a formal expression
$$
\prod\limits_{p \textrm{ prime}} p^{i_p}
$$
with $i_p\in \N_0\cup\{\infty\}$.
Similarly we set
$$
g_Y^\res(f):=\lcm\{\ g_y^\res(f)\ |\ y\in Y\}.
$$
Again $g_Y(f)$ divides $g_Y^\res(f)$ (as supernatural numbers), but in general they are different.
If $f$ is generalizing, then both agree.

\begin{lemma}
\label{lemma-case-n=1-in-criterion-for-effective-cycles-saturated}
Let $f:X\to Y$ be a morphism of noetherian schemes and $y\in Y$ a point, such that every generic point of $f^{-1}(\overline{\{y\}})$ lies over $y$.
Then the subgroup
$$
\frac{1}{g_y(f)}\Z f^*(y)\subseteq \Zcal^*(X)
$$
is saturated, i.e.,\ $\frac{1}{g_y(f)}\Z f^*(y)=\Q f^*(y)\cap \Zcal^*(X)$.
\end{lemma}
\begin{proof}
As every cycle $\Zcal\in \Zcal^*(X)\cap \Q f^*(y)$ is a $\Z$-linear combination of the generic points in $f^{-1}(y)$ the statement follows immediately from the following observation.
For $v:=(n_1,\ldots,n_r)\in \Z^r-\{0\}$ with $d:=\gcd\{n_1,\ldots,n_r\}$ the group $\Z^r\cap\Q v$ is generated by $\frac{1}{d}v$.
\end{proof}

\begin{proposition}
\label{proposition-criterion-for-effective-cycles-saturated}
Let $f:X\to Y$ be as before, i.e.,\ $f$ is surjective such that $X\times_Y X$ is noetherian and \eqref{equation-descent-sequence-for-cycles} is a complex. Then the following hold.
\begin{enumerate}
\item The subgroup $\Zcal^*_\desc(f)\subseteq \Zcal^*(X)$ is saturated.
\item If the subgroup $\Zcal^*_{\eff.\desc}(f)\subseteq \Zcal^*(X)$ is saturated, then $g_Y(f)=1$.
\item If $g_Y^\res(f)=1$, then $\Zcal^*_{\eff.\desc}(f)\subseteq \Zcal^*(X)$ is saturated.
\end{enumerate}
In particular, if $f$ is generalizing, then $\Zcal^*_{\eff.\desc}(f)$ is saturated if and only if $g_Y(f)=1$.

\end{proposition}
\begin{proof}
The group $\Zcal^*_\desc(f)$ is the kernel of a homomorphism of torsion-free abelian groups and hence saturated.
If $y\in Y$ is a point, then
$$
\Zcal:=\frac{1}{g_y(f)}\cycl(f^{-1}(\cl[y]))
$$
is a cycle (with integral coefficients) and $g_y(f)\Zcal\in \Zcal^*_{\eff.\desc}(f)$.
As $f$ is surjective $f^\naive$ is injective by \Cref{proposition-f-surjective-implies-f*-injective} and hence if $\Zcal$ is effective, then necessarily $\Zcal=f^\naive(\frac{1}{g_y(f)}y)$ and thus, $g_y(f)=1$.
This shows that $g_Y(f)=1$ if $\Zcal^*_{\eff.\desc}(f)$ happens to be saturated.

\noindent
Conversely, assume that $g_Y^\res(f)=1$ and let $\Zcal\in \Zcal^*(X)$ be a cycle such that $m\Zcal=f^*(\mathcal{Y})$ for some integer $m\neq 0$ and some (unique) cycle
$$
\mathcal{Y}=\sum\limits_{i=1}^n m_iy_i\in \Zcal^*(Y).
$$
We assume that $m_i\neq 0$ for all $i$, that the $y_i$ are pairwise different and moreover,
$$
\codim(\overline{\{y_1\}})\leq\ldots\leq\codim(\overline{\{y_n\}}).
$$
We argue by induction on the number
$$
s:=\#\{\ f(x_i)\ |\ \Zcal=\sum m_ix_i\ \text{with}\ m_i\neq 0 \}.
$$
If $s=1$, then $n=1$ and \Cref{lemma-case-n=1-in-criterion-for-effective-cycles-saturated} can be applied to conclude $\Zcal\in \Z f^\naive(y_1)$.
In the general case, let $x_1,\ldots,x_r$ be the generic points in $f^{-1}(\cl[y_1])$ and assume that $x_1,\ldots,x_d$ lie over $y_1$ while $x_{d+1},\ldots,x_r$ do not ($r=d$ is allowed).
We define
$$
U:=X-\bigcup\limits_{i=d+1}^r \cl[x_i]
$$
and
$$
j:U\to X
$$ as the natural open immersion of $U$ in $X$.
Furthermore, let
$$
g:Y_1:=\Spec(\O_{Y,y_1})\to Y
$$
be the Zariski localization at $y_1$.
The diagram
$$
\xymatrixcolsep{6pc}
\xymatrix{
\Zcal^*(U\times_{Y} Y_1)  & \Zcal^*(U)\ar[l]_{(\Id_U\times g)^*} \\
\Zcal^*(X\times_{Y} Y_1) \ar[u]^{(j\times \Id_{Y_1})^*} & \Zcal^*(X)\ar[u]_{j^*}\ar[l]_{(\Id_X\times g)^*} \\
\Zcal^*(Y_1) \ar[u]^{(f\times \Id_{Y_1})^\naive} & \Zcal^*(Y)\ar[u]_{f^\naive}\ar[l]_{g^*}
}
$$
is commutative by \Cref{lemma-descent-sequences-compatible-with-zariski-localization}. 
We conclude
$$
(\Id_U\times g)^*j^*(m\Zcal)=m_1(f\circ j\times \Id_{Y_1})^\naive(y_1)
$$
as $y_1$ is not a specialisation of one of the $y_2,\ldots,y_n$ and therefore $g^*(\mathcal{Y})=m_1y_1$.
By the case $n=1$, applied to $(\Id_U\times g)^*j^*(\Zcal)$ and $f\circ j \times \Id_{Y_1}:U\times_Y Y_1\to Y_1$, it follows that $m$ divides $m_1$.
By construction, the induction hypothesis may be applied to
$$
\Zcal':=\Zcal-\frac{m_1}{m}f^*(y_1)
$$
as no point occuring in $\Zcal'$ lies over $y_1$.
Using induction we obtain $\Zcal'\in \Zcal^*_{\eff.\desc}(f)$ and hence $\Zcal\in \Zcal^*_{\eff.\desc}(f)$.
This finishes the proof.
\end{proof}

\medskip\noindent
We give an example to show that in general $g_Y(f)\neq g_Y^\res(f)$ and that $g_Y^\res(f)\neq 1$ is possible even if $\Zcal^*_{\eff.\desc}(f)$ is saturated.

\begin{example}
\label{example-Zeffdesc-saturated-but-gYres-not-1}
Let $Y=\Spec(R)$ be the spectrum of a discrete valuation ring $R$ with generic point $\eta=\Spec (K)$ and special point $s=\Spec(k)$.

Let $\pi\in R$ be a uniformiser and define
$$
X_1:=\Spec(R[t]/(\pi^nt^n,t^m)
$$
with $n,m\geq 1$.
Then as a topological space $X_1:=\{\eta_1,s_1\}$ with $\eta_1$ specialising to $s_1$.
We define
$$
X:=X_1 \coprod \{s\}
$$
and take $f:X\to Y$ as the natural morphism.
Then
$$
f^\naive(\eta)= e\eta_1+s
$$
and
$$
f^\naive(s)=ms_1+s
$$
with $e:=\min\{n,m\}$, hence $\Zcal^*_{\eff.\desc}(f)$ is generated by
$$
e\eta_1+s,e\eta_1-ms_1.
$$
In particular,
$$
{\Zcal^*(X)}/{\Zcal^*_{\eff.\desc}(f)}\cong {\Z\eta_1\oplus\Z s_1}/(e\eta_1-ms_1)\cong\Z\oplus \Z/d
$$
with $d=\gcd\{e,m\}$ is torsionfree, i.e., $\Zcal^*_{\eff.\desc}(f)$ saturated, if and only if $d=1$.
But
$$
g_{\eta}^\res(f)=\length(K[t]/(\pi^nt^n,t^m))=\min\{n,m\}=e\neq 1
$$
for general $n,m$, while $g_Y(f)=1$ for every $n,m$.
\end{example}

\Cref{example-Zeffdesc-saturated-but-gYres-not-1} shows, that in the statement of \Cref{proposition-criterion-for-effective-cycles-saturated} the number $g_Y(f)$ can not be replaced by $g_Y^\res(f)$.
It also shows that descent of algebraic cycles does not hold, as then $\Zcal^*_{\eff.\desc}(f)$ must be separated.

\medskip\noindent
We proceed with the question under which conditions on $f$ the group $H_f$ is zero.
Examples with $H_f\neq 0$, additionally to \Cref{example-Zeffdesc-saturated-but-gYres-not-1}, are the following. Firstly, local artinian schemes $X$ over a field $Y=\Spec(k)$ (such that $X\times_Y X$ is noetherian). Secondly, set
$$
Y=\Spec(k[t]), X=\Spec(k[t,x]/(x^n-t))
$$
and
$$f:X\to Y: (t,x)\mapsto t
$$
the natural projection.
The morphism $f$ is faithfully flat with $n\mid g_Y(f)$ as the fibre over $t=0$ is
$$
\Spec(k[x]/(x^n)).
$$
The cycle $\Zcal:=\cycl(\Spec(k[t,x]/(t,x)))$ is a cycle with descent datum, but
$$
m\Zcal=f^*(\cycl(\Spec(k[t]/(t)))
$$
has an effective descent datum for $m\in \Z$ only if $n\mid m$.

\medskip\noindent
The reason that descent can fail for generalizing morphisms is basically that two subschemes can have the same cycle without being equal, hence $\Zcal^*_\desc(f)$ is in some sense ``too large''.
Concretely, in the example $Y=\Spec(k)$ and $X=\Spec(k[t]/(t^2))$ consider the subscheme $X_{\red}\subseteq X$.
Then $\pr^{-1}_1(X_{\red})\neq \pr^{-1}_2(X_{\red})$, but both schemes have length $2$ with same support. In particular, their cycles agree.

\medskip\noindent
As another numerical invariant we define $\pi_Y^\res(f)$ as the product\footnote{ in the sense of supernatural numbers} of the $g_y^\res(f)$ over all $y\in Y$.
We arrive at our first condition on $f$ guaranteing descent of cycles (for rather trivial reasons).

\begin{proposition}
\label{proposition-descent-along-universally-bijective-morphisms}
Assume that $f$ is universally bijective and that $Y$ is of finite Krull dimension.
Then $H_f$ is a torsion group annihilated by $\pi_Y^\res(f)$\footnote{ by which we mean that every element in $H_f$ is annihilated by some natural number $\neq 0$ dividing $\pi_Y^\res(f)$}.
\end{proposition}
\begin{proof}
By \Cref{lemma-descent-sequence-is-a-complex-for-universally-bijective-morphisms}, we conclude that $\Zcal^*(X)=\Zcal^*_\desc(f)$.
Let $x\in X$ be point.
We prove $g_Y^\res(f)x\in \Im(f^\naive)$ by induction on $d:=\dim \overline{\{f(x)\}}$.
More precisely, our induction hypothesis states that $x$ is annihilated by
$$
\prod\limits_{y \in \overline{\{f(x)\}}}g_y(f).
$$
If $d=0$, then $f^\naive(f(x))=g_{f(x)}^\res(f)x$ and the claim is established.
For general $x\in X$ we can write by induction
$$
f^\naive(f(x))=g_{f(x)}^\res(f)x+ \Zcal
$$
for some $\Zcal\in \Zcal^*(X)$ with $n\Zcal \in \Im(f^\naive)$ and $n$ dividing
$$
\prod\limits_{x\neq y \in \overline{\{f(x)\}}}g_y(f)
$$
because $x$ is the only point lying over $f(x)$.
In particular, $ng_{f(x)}^\res(f)x\in \Im(f^\naive)$ and the proof is finished.
\end{proof}

\medskip\noindent
Now, we come to our main theorem giving a general condition on $f$ such that descent of algebraic cycles holds for $f$.

\begin{theorem}
\label{theorem-main-theorem}
Assume that $f:X\to Y$ is a universally generalizing morphism between noetherian schemes such that $X\times_Y X$ is again noetherian. Then the following hold.
\begin{enumerate}
\item The group $H_f$ is torsion annihilated by $g_Y(f)$.
\item The descent sequence
$$
\xymatrixcolsep{4pc}
\xymatrix{
0\ar[r] & \Zcal^*(Y)\ar[r]^{f^\naive} & \Zcal^*(X)\ar[r]^{\pr_1^\naive-\pr_2^\naive} &  \Zcal^*(X\times_Y X)
}
$$
is exact if and only if $f$ is surjective and $g_Y(f)=1$.
\end{enumerate}
\end{theorem}

\begin{proof}
The second claim follows directly from the first and from \Cref{proposition-f-surjective-implies-f*-injective}.
To prove the first claim assume that
$$
\Zcal=\sum\limits_{i=1}^n m_ix_i \in \mathcal{Z}^*(X)
$$
is a cycle with $\textrm{pr}_1^*(\Zcal)=\textrm{pr}_2^*(\Zcal)$.
As $f, \pr_1, \pr_2$ are generalizing we may assume by \Cref{proposition-generalizing-morphisms-preserve-grading-by-codimension} that $\Zcal\in \mathcal{Z}^r(X)$ for some $r$, hence $$\textrm{codim}(\overline{\{x_i\}},X)=r$$ for every $i=1,\ldots,n$.
Even further, as the decompositions
$$
\Zcal^*(X)=\bigoplus\limits_{y\in Y}\bigoplus\limits_{x\in f^{-1}(y)}\Z x
$$
and
$$
\Zcal^*(X\times_Y X)=\bigoplus\limits_{y\in Y}\bigoplus\limits_{z\in \pr_1^{-1}f^{-1}(y)}\Z z
$$
are preserved under $\pr_1^*$ and $\pr_2^*$, we may further assume that
$x_1,\ldots,x_n$ lie in the fibre $f^{-1}(y)$ for some $y\in Y$.
Let $Y'=\cl[y]$ be the (reduced) closure of $y$ and denote by $i:Y'\to Y$ the natural closed immersion.
Let $i':X'\to X$ be the base change of $i$ along $f$.
Then, by \Cref{proposition-push-forward-commutes-with-naive-pull-back}, the diagram
$$
\xymatrix{
0\ar[r]&\Zcal^*(Y')\ar[r]\ar@{^{(}->}[d]^{i_*}&\Zcal^*(X')\ar[r]\ar@{^{(}->}[d]^{i'_*}&\Zcal^*(X'\times_{Y'} X')\ar@{^{(}->}[d]^{({i'}\times {i'})_*}\\
0\ar[r]&\Zcal^*(Y)\ar[r]&\Zcal^*(X)\ar[r]&\Zcal^*(X\times_{Y} X)
}
$$
of descent sequences with all vertical arrows injective commutes. 
As $\Zcal$ is the image under $i'_*$ of an element with descent datum, we can replace $Y$ by $Y'$ and $X$ by $X'$ and assume that $y\in Y$ is generic.
Next we want to reduce to the case that $Y=\Spec(k(y))$ is a field.
Let $j:Y':=\Spec(k(y))\to Y$ be the inclusion and $j'$ its base change to $X$. Then the diagram
$$
\xymatrix{
0\ar[r] & \Zcal^*(Y)\ar[r]\ar[d]^{j^*} & \Zcal^*(X)\ar[r]\ar[d]^{{j'}^*} & \Zcal^*(X\times_{Y} X)\ar[d]^{(j'\times j')^*}\\
0\ar[r] & \Zcal^*(Y')\ar[r] & \Zcal^*(X')\ar[r] & \Zcal^*(X'\times_{Y'} X')
}
$$
commutes by \Cref{lemma-descent-sequences-compatible-with-zariski-localization}.
If ${j'}^*(\Zcal)$ is the image of some $\mathcal{Y}=my\in\Zcal^*(Y')$, then $\Zcal=f^\naive(my)$ as the generic fibre of $f$ agrees with $X'$ and $f$ is generalizing.
In other words, we may assume that
$$
Y=\textrm{Spec}(k),\ k:=k(y),
$$
is a point.
We next show that (if $Y$ is a point) the generic points of $X$ are precisly the $x_i$ (recall $\Zcal=\sum m_ix_i$). Let $\eta\in X$ be a generic point of $X$.
By the assumption $\pr_1^*(\Zcal)=\pr_2^*(\Zcal)$ it follows that $$\coprod\limits_{j=1}^n x_j\times_Y X = \coprod\limits_{j=1}^n X\times_Y x_j.$$
As $Y$ is a point, for every $j$ there exist a point $z\in X\times_Y x_j$ with $\pr_1(z)=\eta$ showing that
$$
\eta\in \textrm{pr}_1(\coprod\limits_{j=1}^n X\times_Y x_j)=\pr_1(\coprod\limits_{j=1}^n x_j\times_Y X)=\{x_1,\ldots,x_n\}.
$$
We assumed that there are no specialisations among the $x_i$, hence conversely every $x_i$ is generic.
Let
$$
g:X':=\coprod\limits_{i=1}^n\textrm{Spec}(\mathcal{O}_{X,x_i})\to X
$$
be the natural inclusion, which is a flat morphism.
By \Cref{proposition-functoriality-for-pull-back}, the diagram
$$
\xymatrix{
0\ar[r] & \Zcal^*(Y)\ar[r]\ar[d]^{\Id_Y} & \Zcal^*(X)\ar[r]\ar[d]^{g^*} & \Zcal^*(X\times_{Y} X)\ar[d]^{(g\times g)^*}\\
0\ar[r] & \Zcal^*(Y)\ar[r] & \Zcal^*(X')\ar[r] & \Zcal^*(X'\times_{Y'} X')
}
$$
commutes.
By construction of $X'$, the morphism $g^*$ induces isomorphisms
$$
g^*:\Zcal^*_{\eff.\desc}(f)\to \Zcal^*_{\eff.\desc}(g\circ f)
$$
and
$$
g^*:\Zcal^*_\desc(f)\to \Zcal^*_\desc(g\circ f).
$$
Hence we may replace $X$ by $X'=\coprod\limits_{i=1}^n\Spec(\O_{X,x_i})$.
Let $R_i:=\mathcal{O}_{X,x_i}$ with residue field $k_i$ and $r_i:=\textrm{length}_{R_i}(R_i)$.
We first show that $m_i/r_i$ is independant of $i$.
We compute (using \Cref{lemma-cycles-of-sheaves-are-additive-in-short-exact-sequences})
$$
\pr_1^*(\Zcal)=\sum\limits_{i=1}^n m_i\sum\limits_{j=1}^n\cycl(k_i\otimes_k R_j)=\sum\limits_{i=1}^n m_i\sum\limits_{j=1}^n r_j\cycl(k_i\otimes_k k_j),
$$
which by assumption equals
$$
\pr_2^*(\Zcal)=\sum\limits_{i=1}^n m_i\sum\limits_{j=1}^n r_j\cycl(k_j\otimes_k k_i).
$$
We can now conclude
$$
m_ir_j=m_jr_i
$$
for all $i,j=1,\ldots,n$ as the supports of $\Spec(k_i\otimes_k k_j)\subseteq X\times_Y X$ are disjoint and thus the cycles $\cycl(k_i\otimes_k k_j), i,j=1,\ldots,n,$ are linearly independent.
We obtained that with $q:=\frac{m_1}{r_1}$ the cycle
$$
\Zcal=\sum\limits_{i=1}^n m_ix_i=q\sum\limits_{i=1}^n r_ix_i=qf^\naive(y)
$$
lies in $\Q f^\naive(y)\cap \Zcal^*(X)=\frac{1}{g_y(f)}\Z f^*(y)$ and hence
$$
g_y(f)\Zcal\in \Im(f^*),
$$
which finishes the proof.
Indeed, in the general situation the group
$$
H_f=\Zcal^*_\desc(f)/{\Zcal^*_{\eff.\desc}(f)}
$$
is the direct sum over $y\in Y$ of the groups
$$
H_{y,f}:=((\bigoplus\limits_{x\in f^{-1}(y)}\Z x)\cap \Zcal^*_\desc(f))/{\Z f^*(y)},
$$
as $f$ is universally generalizing, and we proved that for $y\in Y$ the direct summand $H_{y,f}$ is annihilated by $g_y(f)$.
Hence, $H_f=\bigoplus\limits_{y\in Y} H_{y,f}$ is annihilated by
\[
g_Y(f)=\lcm\{\ g_y(f)\ |\ y\in Y\}.\qedhere
\]
\end{proof}

\noindent
Finally we want to remark, that \Cref{theorem-main-theorem} implies that for every scheme $X$, separated and of finite type over a field $k$, the presheaf with transfers $\Z_\tr(X)$ represented by $X$ (see \cite[Definition 2.8]{LectNotesOnMotCohom}) is a sheaf for the \'{e}tale topology, i.e., for every \'{e}tale surjection $f:Y\to Z$ of smooth, separated schemes over $k$ the sequence
$$
0\longrightarrow \Z_\tr(X)(Z)\overset{f^*}{\longrightarrow} \Z_\tr(X)(Y)\overset{\pr_1^*-\pr_2^*}{\longrightarrow} \Z_\tr(X)(Y\times_Z Y)
$$
is exact.

\begin{corollary}
\label{corollary-presheaf-with-transfers-is-a-sheaf}
Let $k$ be a field and let $X$ be a separated scheme of finite type over $k$.
Then the presheaf with transfers $\Z_\tr(X)$ is a sheaf for the \'{e}tale topology.
\end{corollary}
\begin{proof}
Let $Y\to Z$ be an \'{e}tale surjection of smooth, separated schemes over $k$.
By \cite[Lemma 3.3.12]{RelativeCycles} the pull-back
$$
f^*:\Z_\tr(X)(Z)\to \Z_\tr(X)(Y)
$$ of relative cycles is induced by the pull-back of absolute cycles
$$
(\Id_X\times f)^*:\Zcal^*(X\times_k Z)\to \Zcal^*(X\times_k Y)
$$
along the inclusions $\Z_\tr(X)(Z)\subseteq \Zcal^*(X\times_k Z)$ resp.\ $\Z_\tr(X)(Y)\subseteq \Zcal^*(X\times_k Y)$.
By \Cref{theorem-main-theorem} descent of algebraic cycles holds for $\Id_X\times f$ and therefore it suffices to show that for every $w\in X\times_k Z$ the subscheme $\cl[w]$ is finite and surjective over a component of $Z$ if and only if the components of the cycle $f^*(w)\in \Zcal^*(X\times_kY)$ are finite and surjective over a component of $Y$.
But this last statement follows as the properties of being finite and dominant over a component descent along quasi-compact faithfully flat morphisms.
\end{proof}

\bibliography{bibliography}

\providecommand{\bysame}{\leavevmode\hbox to3em{\hrulefill}\thinspace}
\providecommand{\MR}{\relax\ifhmode\unskip\space\fi MR }
\providecommand{\MRhref}[2]{%
  \href{http://www.ams.org/mathscinet-getitem?mr=#1}{#2}
}
\providecommand{\href}[2]{#2}
\begin{thebibliography}{MVW06}

\bibitem[Ful98]{IntersectionTheory}
William Fulton, \emph{Intersection theory}, second ed., Ergebnisse der
  Mathematik und ihrer Grenzgebiete. 3. Folge. A Series of Modern Surveys in
  Mathematics [Results in Mathematics and Related Areas. 3rd Series. A Series
  of Modern Surveys in Mathematics], vol.~2, Springer-Verlag, Berlin, 1998.
  \MR{1644323 (99d:14003)}

\bibitem[GD71]{EGA_I_second}
A.~Grothendieck and J.~A. Dieudonn{\'e}, \emph{El\'ements de g\'eom\'etrie
  alg\'ebrique. {I}}, Grundlehren der Mathematischen Wissenschaften
  [Fundamental Principles of Mathematical Sciences], vol. 166, Springer-Verlag,
  Berlin, 1971. \MR{3075000}

\bibitem[Gro66]{EGA.IV.3}
A.~Grothendieck, \emph{\'{E}l\'ements de g\'eom\'etrie alg\'ebrique. {IV}.
  \'{E}tude locale des sch\'emas et des morphismes de sch\'emas. {III}}, Inst.
  Hautes \'Etudes Sci. Publ. Math. (1966), no.~28, 255. \MR{0217086 (36 \#178)}

\bibitem[MVW06]{LectNotesOnMotCohom}
Carlo Mazza, Vladimir Voevodsky, and Charles Weibel, \emph{Lecture notes on
  motivic cohomology}, Clay Mathematics Monographs, vol.~2, American
  Mathematical Society, Providence, RI; Clay Mathematics Institute, Cambridge,
  MA, 2006. \MR{2242284 (2007e:14035)}

\bibitem[SV00]{RelativeCycles}
Andrei Suslin and Vladimir Voevodsky, \emph{Relative cycles and {C}how
  sheaves}, Cycles, transfers, and motivic homology theories, Ann. of Math.
  Stud., vol. 143, Princeton Univ. Press, Princeton, NJ, 2000, pp.~10--86.
  \MR{1764199}

\end{thebibliography}
\bibliographystyle{amsalpha}

\end{document}